\documentclass[12pt]{article}

\usepackage{amssymb,amsmath}
\usepackage{supertabular,mathdots}
\usepackage{amsbsy}
\usepackage{amscd}
\usepackage{amsthm}
\usepackage{mathrsfs}
\usepackage{verbatim}
\usepackage[colorlinks]{hyperref}

\usepackage{authblk}

\newtheorem{Theorem}[equation]{Theorem}
\newtheorem{Corollary}[equation]{Corollary}
\newtheorem{Lemma}[equation]{Lemma}
\newtheorem{Proposition}[equation]{Proposition}

\theoremstyle{definition}
\newtheorem{Definition}[equation]{Definition}
\newtheorem{Example}[equation]{Example}

\theoremstyle{remark}

\newtheorem{Remark}[equation]{Remark}

\numberwithin{equation}{section}
\numberwithin{figure}{section}

\newcommand{\PP}{{\mathbb P}}

\newcommand{\Q}{{\mathbb Q}}

\newcommand{\N}{{\mathbb N}}

\newcommand{\mb}[1]{\mathbb{#1}}
\newcommand{\mc}[1]{\mathcal{#1}}
\newcommand{\mt}[1]{\text{ #1}}
\newcommand{\mf}[1]{\mathfrak{#1}}

\newcommand{\taban}[1]{ \lfloor #1 \rfloor }

\DeclareMathOperator{\D}{D}

\newcommand{\flag}{\mathscr{B}}

\DeclareMathOperator{\Sym}{Sym}
\DeclareMathOperator{\Skew}{Skew}
\DeclareMathOperator{\Mat}{Mat}

\begin{document}

\title{Unipotent Invariant Matrices}
\author[1]{Mahir Bilen Can}
\author[2]{Roger Howe}
\author[3]{Michael Joyce}
\affil[1]{Department of Mathematics, Tulane and Yale Universities}
\affil[2]{Department of Mathematics, Yale University}
\affil[3]{Department of Mathematics, Tulane University}

	\date{\today}
	\maketitle

\begin{abstract}
We describe the variety of fixed points of a unipotent operator acting on the space of matrices.  
We compute the determinant and the rank of a generic (symmetric, or anti-symmetric) matrix in the fixed variety, 
yielding information about the generic singular locus of the corresponding bilinear form.
\end{abstract}


{\em MSC-2010: 15B10, 05E18, 22E46}

\section{\textbf{Introduction}}

The theory of Springer fibers, perceived by Grothendieck and developed in the papers \cite{Springer76}, 
\cite{Springer78} of Springer, as well as in the papers of Steinberg \cite{Steinberg76} and Spaltenstein \cite{Spaltenstein77} 
is a pinnacle of geometric representation theory. Through the work of Lusztig and others it reflects the most 
intricate aspects of algebraic combinatorics. See \cite{Lusztig81}, for example. 
Concisely, a {\em Springer fiber} associated with a unipotent element $u$ in an algebraic group $G$ 
with a Borel subgroup $B\subset G$ is a closed subset of the form $X^u=\{ F\in G/B:\ u \cdot F = F \}$.
In fact, a Springer fiber is a fiber in a resolution of singularities of the unipotent variety of $G$.

In the main special case when the group is $G=\mt{SL}_n$ over an algebraically 
closed field $\mb{K}$ of characteristic $0$, it is possible to analyze the underlying set of points of $X^u$ explicitly, 
using flags of vector subspaces. 
Let $\flag$ denote the flag variety comprised of flags of the form 
$$
\mc{F}: V_1 \subset V_2 \subset  \dots \subset V_{n-1} \subset V_n = \mb{K}^n,\ \text{where}\ \dim V_i = i\ (i=1,\dots, n).
$$ 
Then for each $i\geq 0$, the cohomology space $H^{i}(\flag^u, \Q )$ affords a representation of the 
symmetric group $S_n$, and the top non-vanishing cohomology space is an irreducible $S_n$-module 
(\cite{Springer76}, \cite{Springer78}). Furthermore, all irreducible components of $\flag^u$ have the same dimension 
$e(u)$ which is equal to the half of the dimension of the top non-vanishing cohomology space 
$H^{2e(u)} (\flag^u,\Q)$, \cite{Spaltenstein77}. These leads to striking combinatorial results as 
in \cite{HottaShimomura,Steinberg88}.

In this paper, by replacing flags of vector subspaces with matrices, we investigate the fixed locus of a unipotent operator 
acting on spaces of bilinear forms. In some sense this is the approach of classical invariant theory, performing 
computations on a big open set. In fact, the purpose of our undertaking is two-fold, first of which is 
to compute explicitly the highest weight vectors of the Lie algebra $\mf{sl}_2$ acting on the space of 
matrices via the induced action of an embedding $SL_2 \hookrightarrow SL_n$, where $g \in SL_n$ acts by $g \cdot A = (g^{-1})^\mathsf{T} A g^{-1}$.

The second is to prepare the ground for extending the above mentioned combinatorial 
results to the case of certain compactifications of classical symmetric pairs $\mt{SL}_{n} / \mt{SO}_n$ and 
$\mt{SL}_{2n} / \mt{Sp}_{2n}$. 
Here, $\mt{SO}_n$ is the special orthogonal group consisting of operators $x\in \mt{SL}_{2n}$ 
satisfying $x x^\top = id_n$, and $\mt{Sp}_{2n} \subset \mt{SL}_{2n}$ is the symplectic subgroup consisting of 
operators preserving the skew-symmetric form 
\begin{align}\label{the skew form}
J = 
\begin{pmatrix}
0 & id_n \\
-id_n & 0
\end{pmatrix},
\end{align}
where $id_n$ is the $n\times n$ identity matrix.
In other words, $\mt{Sp}_{2n} = \{g\in \mt{SL}_{2n}:\ g^\top J g = J \}$.

As is mentioned, some deeper aspects of Springer-fibers arise in the context of character sheaves of Luzstig in which 
certain equivariant compactifications of $G$ are in need, \cite{Lusztig04}, \cite{HeLusztig06}. 
Building on the work of Lusztig, in the articles \cite{He06}, \cite{HeThomsen06}, 
He and He-Thomsen investigate the unipotent variety and its subvarieties in a wonderful compactification of $G$. 
See Springer's I.C.M. report \cite{Springer06}, also. 
Although we do not get into wonderful compactifications in this manuscript, we already have some extensions of our work 
to this setting, \cite{CJ10}.

We organized our paper as follows. In Section \ref{S:Notation} we set our notation. 
In Section \ref{S:Preliminaries} we present preliminary results, in particular we find a matrix equation $U(u,x)=0$ 
determining if a matrix $x\in \Mat_n$ is fixed by the unipotent operator $u\in \mt{SL}_n$, or not.

In Section \ref{S:Fixed points}, for a fixed $u$, we describe explicitly the locus of $U(u,x)=0$. In particular, we 
compute the dimensions of the fixed loci in $\Skew_n$ and in $\Sym_n$.
In Section \ref{S:sl_2} we recognize each fixed locus as an $\mf{sl}_2$-module with its highest weight spaces
given by the matrix blocks as calculated in Section \ref{S:Fixed points}.

In Section \ref{S:Non-degenerate} we find an intriguing determinantal formula for a unipotent fixed matrix. 
Finally, in Section \ref{S:Rank} we determine exactly which symmetric (respectively skew-symmetric) unipotent 
fixed matrix is non-degenerate. 

\vspace{1cm}

\textbf{Acknowledgements} The first author is partially supported by the Louisiana Board of Regents
enhancement grant.  We thank Joseph Silverman for helpful conversations and the anonymous 
referee for their careful reading and suggestions for improvement.

\section{\textbf{Notation}}
\label{S:Notation}

We use $\N$ to denote non-negative integers, and if $n\in \N$ is non-zero, 
then we denote by $[n]$ the set $\{1,2,\dots, n\}$.

Throughout we denote by $\mathbb{K}$ an algebraically closed field of characteristic $0$ and 
make the assumption that all vector spaces are over $\mb{K}$. Our basic list of notation is as follows:

\begin{eqnarray*}
\Mat_n: &&\ \text{space of}\ n\times n \ \text{matrices}\\
\mt{SL}_n \subset \Mat_n: &&\ \text{invertible matrices with determinant 1}\\
\Sym_n: &&\ \text{space of}\ n\times n \ \text{symmetric matrices}\\
\Sym_n^0 \subset \Sym_n: &&\  \text{invertible symmetric matrices }\\
\Skew_n: &&\ \text{space of}\ n\times n \ \text{skew-symmetric matrices}\\
\Skew_n^0 \subset \Skew_n: &&\  \text{invertible skew-symmetric matrices }\\
\end{eqnarray*}
Note that $\Mat_n \cong \Sym_n \bigoplus \Skew_n$.

There is a common left action of $\mt{SL}_n$ on the spaces $\Mat_n,\Sym_n$ and $\Skew_n$ given by 
\begin{align}\label{Main Action}
g \cdot A = (g^{-1})^\mathsf{T} A g^{-1}.  
\end{align}
It follows from Cholesky's decompositions that on $\Skew_n^0$ and $\Sym_n^0$ the action (\ref{Main Action}) of $\mt{SL}_n$ 
is transitive (see Appendix B of \cite{GoodmanWallach}). 
It is easy to check that the stabilizer of $J \in \mt{Skew}_{2n}^0$ ($J$ is as in (\ref{the skew form})) is 
the symplectic subgroup $\mt{Sp}_{2n} \subset \mt{SL}_{2n}$, and the stabilizer of the identity element 
$id_n \in \mt{Sym}_n^0$ is the special orthogonal group $\mt{SO}_n:=\{ A \in \mt{SL}_n:\ A A^\top = id_n \}$. 
Thefore, the quotients $\mt{SL}_{2n}/\mt{Sp}_{2n}$ and $\mt{SL}_{n}/\mt{SO}_{n}$ are canonically identified with 
the affine varieties $\mt{Skew}_{2n}^0$ and $\mt{Sym}_n^0$, respectively.

A partition of $n$ is a non-increasing sequence of non-negative integers that sum to $n$. 
Let $\lambda = (\lambda_1, \lambda_2, \dots, \lambda_k)$ be a partition of $n$ with $\lambda_k \geq 1$.
In this case, we call $k$ the length of $\lambda$, and denote it by $\ell(\lambda)$. Entries $\lambda_i$, $i=1,\dots, k$
are called the {\em parts} of $\lambda$.
Alternatively, we write $\lambda = (1^{\alpha_1}, 2^{\alpha_2}, \dots, l^{\alpha_l})$ to indicate that 
$\lambda$ consists of $\alpha_1$ 1's, $\alpha_2$ 2's, and so on.  
Terms with zero exponent may be added and removed without altering $\lambda$.  
For example, each of $(3,3,2)$, $(3,3,2,0)$, $(1^0, 2^1, 3^2)$, and $(2^1, 3^2)$ represent the same 
partition $3 + 3 + 2$ of $8$.

Given a partition $\lambda$ of $n$ with $k$ parts, the {\em $\lambda$-decomposition} of an $n$-by-$n$ matrix 
is obtained by inserting horizontal lines after rows $\lambda_1$, $\lambda_2$, \dots, $\lambda_{k-1}$ 
and similarly inserting vertical lines after columns $\lambda_1$, $\lambda_2$, \dots, $\lambda_{k-1}$, 
thereby giving a block decomposition of the matrix. For example, here is the $5$-by-$5$ identity matrix $I_5$ 
with its $(2,1,1,1)$-decomposition:
$$
I_5 = \begin{pmatrix}
1 & 0 & \vline & 0 & \vline & 0 & \vline & 0 \\
0 & 1 & \vline & 0 & \vline & 0 & \vline & 0 \\
\hline
0 & 0 & \vline & 1 & \vline & 0 & \vline & 0 \\
\hline
0 & 0 & \vline & 0 & \vline & 1 & \vline & 0 \\
\hline
0 & 0 & \vline & 0 & \vline & 0 & \vline & 1
\end{pmatrix}
$$

\section{\textbf{Preliminaries}}
\label{S:Preliminaries}

\begin{Lemma}\label{lem:fixed point}
Let $G$ be either the additive group  $\mathbb{G}_a = \mathbb{K}^+$ 
or the multiplicative group $\mathbb{G}_m= \mathbb{K}^\times$.
Let $X$ be a complete variety on which $G$ acts,
and let $g \in G$ be an element of infinite order.
Then $X^g = X^G$.
\end{Lemma}

	
\begin{proof}
Clearly $X^G \subset X^g$.  Suppose $x \in X^g$. Then $x \in X^{g^n}$ for any $n \in \mathbb{Z}$. 
Consider the map $G \rightarrow X$ given by $t \mapsto t \cdot x$.  Since $X$ is complete, 
this map extends to a morphism $\phi: \mathbb{P}^1\rightarrow X$. 
Since $\text{char } \mathbb{K} = 0$, $\phi^{-1}(x)$ is infinite, and therefore the image of $\phi$ is a point.  Thus, $x \in X^G$.
\end{proof}

\begin{Corollary}\label{cor:fix pt}
Let $N$ be a nilpotent matrix with entries in $\mathbb{K}$, $u = \exp N$, 
and $U = \{ \exp(t N) : t \in \mathbb{K} \}$. 
If $X$ is any complete variety on which $U$ acts, then $X^u = X^U$.
\end{Corollary}

\begin{Proposition}\label{P:main condition} 
Let $N$ be a nilpotent matrix and let $u = \exp(N)$. For $A\in \Mat_n$ 
the following are equivalent:
\begin{enumerate}
\item $A$ is fixed by $u$;
\item $AN + N^\mathsf{T} A = 0$.
\end{enumerate}
\end{Proposition}

\begin{proof}
Consider the one-dimensional unipotent subgroup of $\mt{SL}_n$ given by
$$
U = \{ \exp(tN) : t \in \mathbb{K} \}.
$$
Clearly $U$ acts on the projective space $\PP(\Mat_n)$.
By Corollary \ref{cor:fix pt}, $\PP(\Mat_n)^u = \PP(\Mat_n)^U$.

To find fixed points of the subgroup $U$, we seek solutions to the equations
\begin{equation}\label{eq:fixed point condition}
\exp(-t N)^\mathsf{T} A \exp(-t N) = A
\end{equation}
for all $t \in \mathbb{K}$.
Viewing this equation in the ring of $n$-by-$n$ matrices with coefficients in $\mathbb{K}[t]$, 
differentiating with respect to $t$ and then setting $t = 0$, we obtain
\begin{equation}\label{eq:NA+AN^T=0}
AN + N^\mathsf{T} A= 0.
\end{equation}

Conversely, assume that (\ref{eq:NA+AN^T=0}) holds.  Then an easy induction shows that
\begin{equation}\label{eq:N^k formula}
A N^k = (-1)^k (N^{\mathsf{T}})^k A
\end{equation}
for all $k \geq 0$.  Expanding $\exp(t N)$ as a polynomial in $N$ and using (\ref{eq:N^k formula}) gives
$$
A \exp(t N)  = \exp(-t N^{\mathsf{T}}) A,
$$
which is equivalent to (\ref{eq:fixed point condition}).
\end{proof}

\begin{Remark}
Note that the Jordan type of $u = \exp(N)$ is the same as
the Jordan type of $N$, as a simple row reduction argument shows.
\end{Remark}

We state the following elementary result without proof.

\begin{Lemma}\label{lem:conjugation}
Suppose that $N$ and $N'$ are two nilpotent matrices that are conjugate in $\mt{SL}_n$, say $N' = S N S^{-1}$.
Define $u = \exp(N)$ and $u' = \exp(N')$.  Then $u' = S u S^{-1}$ and the fixed point loci of $u$ and $u'$ are isomorphic via $A \mapsto (S^{-1})^{\mathsf{T}} A S^\mathsf{T}$.  
\end{Lemma}


It follows from the above discussions that the fixed point space $X^u$ for $X= \Mat_n, \Sym_n$, 
or $X=\Skew_n$ depend only on the Jordan type of $u$. Equivalently, 
it depends only on the Jordan type of a nilpotent matrix $N$ for which $u = \exp(N)$.

The Jordan classes of $n$-by-$n$ nilpotent matrices are in
bijection with the partitions of $n$.  Indeed, let $N_{(p)}$ be the $p$-by-$p$ matrix
$$N_{(p)} = \left( \begin{matrix}
0 & 1 & 0 & \cdots & 0 & 0\\
0 & 0 & 1 & \cdots & 0 & 0\\
\vdots & \vdots & \vdots & \ddots & \vdots & \vdots\\
0 & 0 & 0 & \cdots & 0 & 1\\
0 & 0 & 0 & \cdots & 0 & 0
\end{matrix} \right).$$
Then the above correspondence associates a partition
$\lambda = (\lambda_1, \lambda_2, \dots, \lambda_k$),
to the Jordan matrix $N_{\lambda}$ given in block form by
$$
N_{\lambda} = \left( \begin{matrix}
N_{\lambda_1} & 0 & \cdots & 0\\
0 & N_{\lambda_2} & \cdots & 0\\
\vdots & \vdots & \ddots & \vdots\\
0 & 0 & \cdots & N_{\lambda_k}
\end{matrix} \right).
$$
From now on we write $X^\lambda$ in place of $X^u$, if $\lambda$ is the Jordan type of $u$.

\section{Fixed points}
\label{S:Fixed points}

For $k \in [n]$, let $N_{(k)}$ denote the $k\times k$ principal nilpotent matrix with a single Jordan-block.
Suppose $A = (a_{i,j})_{ {i=1,\dots, k} \atop {j=1,\dots,n}}$ is a generic $k\times n$ matrix.
Then it is easy to verify that $AN_{(n)} + N_{(k)}^\mathsf{T} A= 0$ if and only if $A$ has the following form:
\begin{align}\label{A: k less than n}
 \begin{pmatrix}
0 & \cdots   & 0 & 0 &0 &\cdots &0 & 0 & a_{1,n} \\
0 & \cdots   & 0 & 0 &0 & \cdots  & 0 & -a_{1,n} & a_{2,n} \\
0 & \cdots   & 0 & 0 & 0 & \cdots &  a_{1,n} & -a_{2,n} & a_{3,n} \\
\vdots & \cdots  & \vdots & \vdots & \vdots & \iddots & \vdots & \vdots & \vdots \\
0 & \cdots & 0 & 0 & (-1)^{n-1} a_{1,n} &\cdots & a_{k-4,n} & -a_{k-3,n} & a_{k-2,n}\\
0 &\cdots & 0 & (-1)^na_{1,n} &(-1)^{n-1} a_{2,n}&  \cdots & a_{k-3,n} & -a_{k-2,n} & a_{k-1,n}\\
0 & \cdots & (-1)^{n+1} a_{1,n} & (-1)^n a_{2,n} & (-1)^{n-1} a_{3,n} & \cdots & a_{k-2,n} & -a_{k-1,n} & a_{k,n}
\end{pmatrix}.
\end{align}

If the roles of $k$ and $n$ are interchanged, that is to say, if $k > n$, then 
$AN_{(n)} + N_{(k)}^\mathsf{T} A= 0$ if and only if $A$ has the following form:
\begin{align}\label{A: k greater than n}
\begin{pmatrix}
0  & 0 & \cdots & 0 & 0 &0 & 0 \\
\vdots & \vdots & \vdots & \vdots & \vdots & \vdots& \vdots \\
0  & 0 &\cdots & 0 & 0 & 0 & a_{k,k} \\
0  & 0 &\cdots & 0 & 0 & -a_{k,k} & a_{k,k+1} \\
0  & 0 &\cdots & 0 & a_{k,k} & -a_{k,k+1} & a_{k,k+2} \\
\vdots & \vdots& \iddots & \vdots & \vdots & \vdots & \vdots \\
0 & 0  & \cdots & - a_{k,k+1} & a_{k,k} & -a_{k,k+1} & a_{k,n} \\
0 & (-1)^k a_{k,k+1}  & \cdots & -a_{k,k-1} & a_{k,k} & -a_{k,k+1} & a_{k,n} \\
(-1)^{k+1}a_{k,k}& (-1)^k a_{k,k+1}  & \cdots & -a_{k,k-1} & a_{k,k} & -a_{k,k+1} & a_{k,n} \\
\end{pmatrix}.
\end{align}

For an arbitrary matrix $B=(b_{i,j})_{i,j=1}^n$, let us call the portion of $B$ consisting of $b_{i,j}$ with $i+j = n+k$ the
{\em $k$-th anti-diagonal}. 
Proof of the following propositions are easy, so we omit them.
\begin{Proposition}\label{P:symmetric regular case}
Let $A=(a_{i,j})_{i,j=1}^n$ be a symmetric matrix satisfying $A N_{(n)} + N_{(n)}^\mathsf{T} A= 0$. Then 
\begin{enumerate}
\item if $n=2k+1$ for some $k\in \N$, then for $l\in [k]$, the $2l$-th anti-diagonal of $A$ consists of zeros only,
and furthermore, $a_{2l-1,n},-a_{2l-1,n},\cdots, (-1)^{n-2l+1} a_{2l-1,n}$ are top-to-bottom entries of $(2l-1)$-st anti-diagonal  
of $A$.
\item If $n=2k$ for some $k\in \N$, then for $l\in [k]$, the $(2l-1)$-st anti-diagonal of $A$ consists of zeros only,
and furthermore, $a_{2l,n},-a_{2l,n},\cdots, (-1)^{n-2l} a_{2l,n}$ are top-to-bottom entries of $2l$-th anti-diagonal of $A$.
\end{enumerate}
\end{Proposition}

Similarly, we have 
\begin{Proposition}\label{P:anti-symmetric regular case}
If $A=(a_{i,j})_{i,j=1}^n$ is an anti-symmetric matrix satisfying $A N_{(n)} + N_{(n)}^\mathsf{T} A= 0$, then 
\begin{enumerate}
\item if $n=2k+1$ for some $k\in \N$, then for $l\in[k+1]$, the $(2l-1)$-st anti-diagonal of $A$ consists of zeros only,
and furthermore, $a_{2l,n},-a_{2l,n},\cdots, (-1)^{n-2l} a_{2l,n}$ are top-to-bottom entries of $2l$-th anti-diagonal of $A$.

\item If $n=2k$ for some $k\in \N$, then for $l\in [k]$, the $2l$-th anti-diagonal of $A$ consists of zeros only,
and furthermore, $a_{2l-1,n},-a_{2l-1,n},\cdots, (-1)^{n-2l+1} a_{2l-1,n}$ are top-to-bottom entries of $(2l-1)$-st anti-diagonal  
of $A$.
\end{enumerate}
\end{Proposition}

\begin{Corollary}\label{C:one part partition}
\begin{enumerate}
\item If $X= \Mat_n$, then $X^{(n)} \cong \mb{A}^{n}$. 
\item If $X=\Sym_n$, then $X^{(n)} \cong \mb{A}^{\lfloor(n+1)/2\rfloor}$.
\item If $X=\Skew_n$, then $X^{(n)} \cong \mb{A}^{\lceil(n-1)/2\rceil}$.
\end{enumerate}
\end{Corollary}

\subsection{Fixed points of an arbitrary unipotent element}

Next, we consider the case of a general partition $\lambda = (\lambda_1,\dots, \lambda_k)$ of $n$ with $k$ parts. 

\begin{Lemma}\label{L:block decomposition for matrices}

Let $A \in X^\lambda$ be a generic matrix, where $X=\Mat_n$. Then $A$ has the block decomposition 
\begin{equation}\label{E: general form of a unipotent fixed matrix}
A = \begin{pmatrix}
B_{\lambda_1, \lambda_1} & \vline & B_{\lambda_1, \lambda_2} & \vline & \cdots & \vline & B_{\lambda_1, \lambda_k}\\
\hline
B_{\lambda_2, \lambda_1} & \vline & B_{\lambda_2, \lambda_2}& \vline & \cdots & \vline & B_{\lambda_2, \lambda_k}\\
\hline
\vdots & \vline & \vdots & \vline & \ddots & \vline & \vdots\\
\hline
B_{\lambda_k, \lambda_1}& \vline & B_{\lambda_k, \lambda_2} & \vline & \cdots & \vline & B_{\lambda_k,\lambda_k}
\end{pmatrix},
\end{equation}
where matrices $B_{\lambda_i, \lambda_j}$ have the form given by (\ref{A: k greater than n}), or by (\ref{A: k less than n}).
The variables occurring in the various $B_{\lambda_i, \lambda_j}$'s are all distinct.
\end{Lemma}

The proof of the following corollary is obtained by counting the variables in each block of (\ref{E: general form of a unipotent fixed matrix}).
\begin{Corollary}
Let $X= \Mat_n$. Then the dimension of $X^\lambda$ is $n + 2 \sum_{i=1}^k (i-1) \lambda_i$.
\end{Corollary}

\begin{proof}[Proof of Lemma \ref{L:block decomposition for matrices}]
Let
\begin{equation*}
A = \begin{pmatrix}
A_{\lambda_1, \lambda_1} & \vline & A_{\lambda_1, \lambda_2} & \vline & \cdots & \vline & A_{\lambda_1, \lambda_k}\\
\hline
A_{\lambda_2,\lambda_1} & \vline & A_{\lambda_2, \lambda_2}& \vline & \cdots & \vline & A_{\lambda_2, \lambda_k}\\
\hline
\vdots & \vline & \vdots & \vline & \ddots & \vline & \vdots\\
\hline
A_{\lambda_k, \lambda_1} & \vline & A_{\lambda_k, \lambda_2}& \vline & \cdots & \vline & A_{\lambda_k,\lambda_k}
\end{pmatrix}\ \text{and}\ 
N_\lambda = \begin{pmatrix}
N_{ (\lambda_1}) & \vline & 0& \vline & \cdots & \vline & 0\\
\hline
0& \vline & N_{(\lambda_2)}& \vline & \cdots & \vline & 0\\
\hline
\vdots & \vline & \vdots & \vline & \ddots & \vline & \vdots\\
\hline
0 & \vline & 0& \vline & \cdots & \vline & N_{(\lambda_k)}
\end{pmatrix} 
\end{equation*}
be the $\lambda$-decompositions of $A$ and $N_\lambda$, respectively. It follows from block-matrix multiplication that 
$AN_{\lambda} + N_{\lambda}^\mathsf{T} A= 0$ if and only 
$A_{\lambda_i,\lambda_j} N_{(\lambda_j)} + N_{(\lambda_i)}^{\mathsf{T}} A_{\lambda_i,\lambda_j} = 0$ for all $i,j \in [k]$. Therefore,
it is clear from previous subsection that $A_{\lambda_i,\lambda_j}$ have the form given 
by (\ref{A: k greater than n}), or by (\ref{A: k less than n}). The second claim is clear, also.
\end{proof}

\begin{Example}\label{ex:two cases}
We illustrate Lemma \ref{L:block decomposition for matrices}. The generic element of $X^{(2,2,1,1)}$ is
\begin{equation}\label{eqn:(2,2,1,1)}
A = \begin{pmatrix}
0 & a & \vline & 0 & c & \vline & 0 & \vline & 0 \\
-a & b & \vline & -c & d & \vline & e & \vline & f \\
\hline
0 & g & \vline & 0 & i & \vline & 0 & \vline & 0 \\
-g & h & \vline & -i & j & \vline & k & \vline & l \\
\hline
0 & m & \vline & 0 & p & \vline & r & \vline & s \\
\hline
0 & n & \vline & 0 & q & \vline & t & \vline & w
\end{pmatrix}.
\end{equation}
\end{Example}

\begin{Remark}\label{rem:generic}
We interpret the block decomposition in Lemma \ref{L:block decomposition for matrices} in two ways. 
Let $\mathcal{A}$ be the set of variables that occur in the blocks of (\ref{E: general form of a unipotent fixed matrix}). 
We can either think of (\ref{E: general form of a unipotent fixed matrix}) as an equation defining the elements of $X^{\lambda}$, 
or we can think of it as defining a particular matrix with entries in $\mathbb{K}(\mathcal{A})$, the field of rational functions.  
In the latter case, we say that $A$ is the {\it generic element} of $X^{\lambda}$.
\end{Remark}

Suppose now that $X= \Sym_n$, or $X= \Skew_n$. Let $A\in X^\lambda$ be a generic element. 
In this case, it follows from (\ref{E: general form of a unipotent fixed matrix}) that if $1\leq i\neq j \leq n$, 
then $B_{\lambda_i,\lambda_j}$ uniquely determines $B_{\lambda_j,\lambda_i}$. 
Therefore, by counting the variables in each of the (upper) blocks we arrive at the following
formulas for the dimensions of the fixed subspaces.
\begin{Corollary} Let $\lambda = (\lambda_1,\dots, \lambda_k)$ be a partition of $n$. Then 
\begin{enumerate}
\item $ \dim \Sym_n^\lambda =\sum_{i=1}^k \left\lfloor \frac{\lambda_i + 1}{2} \right\rfloor+ \sum_{i=1}^k (i-1)\lambda_i$,
\item $ \dim \Skew_n^\lambda =\sum_{i=1}^k \left\lceil \frac{\lambda_i - 1}{2} \right\rceil+ \sum_{i=1}^k (i-1)\lambda_i$.
\end{enumerate}
\end{Corollary}

\begin{Example}\label{ex:two cases}
The generic element of $\Sym_6^{(2,2,1,1)}$ is
\begin{equation}\label{eqn:(2,2,1,1)}
M = \begin{pmatrix}
0 & 0 & \vline & 0 & c & \vline & 0 & \vline & 0 \\
0 & a & \vline & -c & b & \vline & e & \vline & h \\
\hline
0 & -c & \vline & 0 & 0 & \vline & 0 & \vline & 0 \\
c & b & \vline & 0 & d & \vline & f & \vline & i \\
\hline
0 & e & \vline & 0 & f & \vline & g & \vline & j \\
\hline
0 & h & \vline & 0 & i & \vline & j & \vline & k
\end{pmatrix},
\end{equation}

while the generic element of $\Skew_6^{(3,2,1)}$ is
\begin{equation}\label{eqn:(3,3,2)}
\widetilde{M} = \begin{pmatrix}
0 & 0 & 0 & \vline & 0 & 0 & \vline & 0\\
0 & 0 & b & \vline & 0 & d & \vline & 0\\
0 & -b & 0 & \vline & -d & c & \vline & f\\
\hline
0 & 0 & d & \vline & 0 & e & \vline & 0\\
0 & -d & -c & \vline & -e & 0 & \vline & g\\
\hline
0 & 0 & -f & \vline & 0 & -g & \vline & 0\\
\end{pmatrix}.
\end{equation}

\end{Example}

\section{Invariant subspaces as $\mf{sl}_2$ modules}
\label{S:sl_2}

Our computations has a conceptual explanation in terms of representation theory of 
$\mf{sl}_2$, the Lie algebra of 2-by-2 traceless matrices over $\mb{K}$.

Let $V$ be a finite dimensional vector space on which $\mt{SL}_n$ acts. 
We consider the corresponding (linearized) action of $\mf{sl}_n$ on $V$. 
Let $u\in \mt{SL}_n$ be a unipotent element and let $N\in \mf{sl}_n$ be the corresponding nilpotent matrix.
By Jacobson-Morozov Theorem we view $N$ as a member of an $\mf{sl}_2$-triple. 
Thus, computing the space of $u$-fixed vectors amounts to computing the space of highest weight
vectors for the corresponding copy of $\mf{sl}_{2}$.

Let $V= \mb{K}^n$. 
There is the natural induced action of $\mt{SL}_n$ on $\mt{End}(V)= V^* \otimes V$. 
Since $V$ is self-dual as an $\mf{sl}_2$-module, we have the corresponding action of $\mf{sl}_2$ on $V\otimes V$.
which descents onto the symmetric and the skew-symmetric forms; 
\begin{align}\label{symmetric and skew-symmetric}
V \otimes V \cong S^2 (V) \oplus \wedge^2 (V).
\end{align}
Indeed, these induced actions agree with the actions of $\mt{SL}_n$ defined 
above in (\ref{Main Action}).

We are ready to prove the main result of this section:
\begin{Theorem}
Let $V$ be an $n$-dimensional vector space, $H$ be a copy of $\mt{SL}_2$ in 
$G=\mt{SL}(V)$ and let $X$ denote one of the following spaces: $V\otimes V^*$, $S^2 (V)$, or $\wedge^2(V)$. 
Then there exists a partition $\lambda = (\lambda_1,\dots, \lambda_r)$ of $n$ (depending only on the decomposition 
of $V$ as an $\mt{SL}_2$-module) such that the number $\alpha_X$ of highest weight vectors 
for the restriction $\mt{Res}^G_H X$ is 
\begin{enumerate}
\item $\alpha_X=n +2 \sum (i-1) \lambda_i$, if $X = V^*\otimes V$; 
\item $\alpha_X=\sum \left\lfloor \frac{\lambda_i + 1}{2} \right\rfloor + \sum (i-1)\lambda_i$, if $X = S^2 (V)$;
\item $\alpha_X=\sum \left\lceil \frac{\lambda_i - 1}{2} \right\rceil + \sum (i-1)\lambda_i$, if $X = \wedge^2 (V)$.
\end{enumerate}
\end{Theorem}

\begin{proof}
Any finite dimensional $\mt{SL}_2$-module is self dual. 
Since $V\otimes V$ decomposes as in (\ref{symmetric and skew-symmetric}), 
the first formula follows from the second and the third. The proof of the third formula is identical to the second, so we skip its details. 

We proceed with the proof of the second formula.
Recall that for each non-negative integer $k$, there is a unique irreducible $\mf{sl}_2$-module denoted by $U_k$
such that $\dim U_k = k+1$.
Suppose the decomposition of $V$ into its irreducible constituents 
is given by $V = \sum_{i=1}^r m_i U_{\lambda_i^0}$, where $\lambda^0_1 \geq \lambda^0_2 \geq \cdots \geq \lambda^0_r$. 
Hence, if we define $\lambda_i = \lambda_i^0 + 1$, then $\lambda = (\lambda_1^{(m_1)},\dots, \lambda_r^{(m_r)})$ becomes a partition of $n$.
(Here, by $ \lambda_i^{(m_i)}$ we mean the repetition of $\lambda_i$ $m_i$-times.)

Recall also that for vector spaces $W$ and $U$ the second symmetric power decomposes as 
\begin{align}\label{A:decomposition of S2 of W plus U}
S^2(W \oplus U) \cong S^2(W) \oplus W\otimes U \oplus S^2(U).
\end{align}
By iterating Equation (\ref{A:decomposition of S2 of W plus U}) twice, we arrive at 
\begin{align}\label{decompose S2}
S^2 (V) &\cong \sum  S^2 (m_i U_{\lambda^0_i}) \oplus \sum_{i < j} m_i m_j U_{\lambda^0_i} \otimes U_{\lambda^0_j} \notag \\
&\cong   \sum m_k S^2 (U_{\lambda^0_k}) \oplus \sum {m_k \choose 2} U_{\lambda^0_k} \otimes U_{\lambda^0_k}  
\oplus \sum_{i < j} m_i m_j U_{\lambda^0_i} \otimes U_{\lambda^0_j}.
\end{align}

We would like to compute the number of highest weight vectors in this sum. To this end, we recall another important fact.  
The Clebsch-Gordon formula states that, for any two non-negative integers $p$ and $q$,  
\begin{align}\label{A: Clebsch-Gordon}
U_p \otimes U_q \cong \sum_{0 \leq j \leq \min (p,q)} U_{p+q-2j}.
\end{align} 
In particular, if $p=q$, then 
\begin{align}\label{A: Clebsch-Gordon equal}
U_p \otimes U_p  &\cong \sum_{0 \leq j \leq p} U_{2(p-j)}  \notag \\ 
&\cong \sum_{0 \leq a \leq \taban{ \frac{p}{2}} } U_{2(p-2a)}  +\sum_{0 \leq b \leq \taban{ \frac{p}{2}} } U_{2(p-2b-1)},
\end{align}
where the first summand in the last equation is $S^2 (U_p)$ and the second summand is $\wedge^2 (U_p)$.

It follows from (\ref{A: Clebsch-Gordon}) and (\ref{A: Clebsch-Gordon equal}), respectively, that 
\begin{enumerate}
\item the number of highest weight vectors in $U_{\lambda^0_i} \otimes U_{\lambda^0_j}$ $(i < j)$ is $\lambda^0_i +1= \lambda_i$;
\item the number of highest weight vectors in $S^2 (U_{\lambda_k^0})$ is $\left\lfloor \frac{ \lambda^0_i}{2} \right\rfloor+1=
\left\lfloor \frac{ \lambda^0_i +2}{2} \right\rfloor=\left\lfloor \frac{ \lambda_i +1}{2} \right\rfloor$. 
\end{enumerate}
In other words, the total number of highest weight vectors is 
\begin{align}\label{Total number of weight vectors}
\sum m_i \left\lfloor \frac{ \lambda_i +1}{2} \right\rfloor + \sum {m_k \choose 2} \lambda_k + \sum_{ i < j} m_i m_j \lambda_j
\end{align}

The proof is finished by observing that each summand in (\ref{Total number of weight vectors}) is equivalent to 
the number of independent variables in an (upper) block of the $\lambda$-decomposition of a generic unipotent 
fixed symmetric matrix of type $\lambda$.


\end{proof}

\section{Non-degenerate unipotent-fixed matrices}
\label{S:Non-degenerate}

It is not immediately evident whether any of the $N_{\lambda}$-fixed matrices are smooth or not.
Our next result allows us to effectively determine this.

\begin{Theorem}\label{thm:det formula}
Let $M$ be the generic element of $\Mat^{\lambda}_n$.
Let $\mu = (\mu_1, \mu_2, \dots, \mu_l)$ be the conjugate partition of $\lambda$,
let $P$ be the matrix obtained by taking only the upper rightmost entry from each block in the $\lambda$-decomposition of $M$,
and, for $1 \leq i \leq l$, let $P_i$ be the upper left $\mu_i$-by-$\mu_i$ submatrix of $P$.  Then
\begin{align}\label{eq:det formula}
\det M = \prod_{i=1}^l \det P_i.
\end{align}
\end{Theorem}

\begin{Example}\label{E: boxed}
Before proving the theorem, we illustrate the idea of the proof with an example.
An element $M$ of $\Mat_6^{(3,3)}$ with its $\lambda$-decomposition is given by
$$
	M = \begin{pmatrix}
	0 & 0 & \framebox{$a$} & \vline & 0 & 0 & \framebox{$d$}\\
	0 & -a & b & \vline & 0  & -d & e\\
	a & -b & c & \vline & d & -e & f\\
	\hline
	0 & 0 & \framebox{$g$} & \vline & 0 & 0 & \framebox{$j$}\\
	0 & -g & h & \vline & 0 & -j & k\\
	g & -h & i & \vline & j & -k & l
	\end{pmatrix},
	$$
where the boxed entries form the matrix $P$ given in Theorem \ref{thm:det formula}.
Since the unboxed entries in the rightmost column of each block are zero,
\begin{equation*}
\det M = \det \begin{pmatrix}
a & d \\
g & j
\end{pmatrix}
\det \begin{pmatrix}
0 & \framebox{$-a$} & \vline & 0 & \framebox{$-d$}\\
a & -b & \vline & d & e\\
\hline
0 & \framebox{$-g$} & \vline & 0 & \framebox{$-j$}\\
g & -h & \vline & j & k
\end{pmatrix}
= \det \begin{pmatrix}
a & d \\
g & j
\end{pmatrix}^3.
\end{equation*}
The boxed entries in the formula give rise to the further factorization in the same way as the initial factorization was obtained.

\end{Example}

\begin{proof}[Proof of Theorem \ref{thm:det formula}]

Let $X$ denote $\Mat_n$ and let $M \in X^\lambda$. 
With respect to its $\lambda$-decomposition place a box around the upper rightmost entry of each block of $M$.  
We define two submatrices of $M$.
First, let $\D_1(M)$ be the $k$-by-$k$ submatrix obtained from the boxed entries; 
let $\D_2(M)$ be the $(n-k)$-by-$(n-k)$ submatrix obtained by removing the rows and columns 
of $M$ that 	contain boxed entries.  
In Example \ref{E: boxed},
$$
\D_1(M) = \begin{pmatrix}
a & d \\
g & j
\end{pmatrix}
\text{ and }
\D_2(M)= \begin{pmatrix}
0 & -a & \vline & 0 & -d\\
a & -b & \vline & d & e\\
\hline
0 &-g & \vline & 0 & -j\\
g & -h & \vline & j & k
\end{pmatrix}.
$$

Because all of the unboxed entries of $M$ in columns containing a boxed entry are $0$, it follows from 
the cofactor expansion that
\begin{equation}\label{E:determinantfactorization}
\det M = (-1)^{n-k} \det \D_1(M) \det \D_2(M),
\end{equation}
where $k = \ell(\lambda)$ is the length of $\lambda$.
Notice that the matrix $-\D_2(M)$ is a unipotent invariant matrix of type $(\lambda_1-1,\lambda_2-1,\dots, \lambda_k-1)$.

Given a partition $\lambda = (\lambda_1, \lambda_2, \dots, \lambda_k)$, 
we construct a finite sequence of $\lambda_1$ partitions inductively by setting 
$\lambda^{(1)}:= \lambda$, and for each $i\geq 1$, defining $\lambda^{(i+1)}$ to be the partition obtained  by 
removing the right-most box from each row of the Young diagram of $\lambda^{(i)}$.
We also consider the series of conjugate partitions $\mu^{(i)}$ of the $\lambda^{(i)}$.
Equivalently, the partition $\mu^{(i)}$ is obtained from $\mu = \mu^{(1)}$ by deleting the top $i - 1$ rows in the Young diagram of $\mu$.  That is, $\mu^{(i)} = (\mu_{i}, \mu_{i+1}, \dots, \mu_l)$.
	
In conjunction with the sequence of partitions constructed above, 
we construct a finite sequence of pairs of submatrices of $M$.  
First, $M^{(1)} = \D_1(M)$ and $ M_{\text{aux}}^{(1)} = \D_2(M)$.  
For $1 \leq i < \lambda_1$, set $M^{(i+1)} = \D_1(M_{\text{aux}}^{(i)})$ and 
$M_{\text{aux}}^{(i+1)} = \D_2(M_{\text{aux}}^{(i)})$.
Then, starting with $M=M^{(1)}$, successive application of the decomposition 
(\ref{E:determinantfactorization}) gives us 
\begin{align}\label{A:signeddeterminat}
\det M = \prod_{i=1}^l (-1)^{|\lambda^{(i)}|-\ell(\lambda^{(i)})} \det M^{(i)},
\end{align}
where $|\lambda^{(i)}|$ is the size, and $\ell(\lambda^{(i)})$ the length, of the partition $\lambda^{(i)}$.
Observe that $|\lambda^{(i)}| - \ell(\lambda^{(i)}) = |\lambda^{(i+1)}| = |\mu^{(i+1))}|$ and that the part $\mu_i$ occurs in $i - 1$ of $\mu^{(2)}, \mu^{(3)}, \dots, \mu^{(l)}$.  Therefore,
\begin{equation*}
\det M =  \prod_{i=1}^l (-1)^{|\mu^{(i+1)}|} \det M^{(i)} = \prod_{i=1}^l (-1)^{(i - 1) \mu_i} \det M^{(i)}.
\end{equation*}

Note that $(i - 1) \mu_i$ is odd if and only if $i$ is even and $\mu_i$ is odd.
At the same time, it follows from the iterated construction of $M^{(i)}$ that $M^{(i)} = \pm P_i$,
with $M^{(i)} = -P_i$ if and only if $i$ is even.
It follows that $\det M^{(i)} = - \det P_i$ precisely when $i$ is even and $\mu_i$ is odd,
yielding (\ref{eq:det formula})
\end{proof}

\section{Rank of a unipotent fixed bilinear form}
\label{S:Rank}

The determinant formula proved in the previous section tells us if a unipotent form is degenerate or not.
When $\Sym_n^{\lambda}$, or $\Skew_n^\lambda$ do not contain any non-degenerate element, we are able to specify 
the rank of a generic $N_{\lambda}$-fixed form.

\begin{Theorem}\label{thm:rank formula}
The corank of a generic element of $\Sym^{\lambda}_n$ is equal to the number of even parts which 
appear an odd number of times in $\lambda$.
Similarly, the corank of a generic element of $\Skew^{\lambda}_n$ is equal to the number of odd parts which appear an odd number
of times in $\lambda$.
\end{Theorem}

Although the proof of the case of skew-symmetric forms is virtually the same as that of the symmetric forms, 
some definitions need to be modified. Therefore, instead of writing a separate proof, we prove the first claim only and 
indicate the places that need a modification in order for it to work for the latter.

The first claim of Theorem \ref{thm:rank formula} follows immediately from Lemmas \ref{lem:rank upper bound} 
and \ref{lem:rank lower bound} below.
To facilitate their proofs, we start with a definition.

\begin{Definition}

For symmetric forms, the {\it degeneracy number} of $\lambda$, 
denoted by $d(\lambda)$, is defined to be the number of even parts which appear an odd number of times in $\lambda$.  
Recall our alternative notation for a partition $\lambda = (1^{\alpha_1}, 2^{\alpha_2}, \dots, l^{\alpha_l})$, 
indicating that the parts of $\lambda$ are $\alpha_1$ 1's, $\alpha_2$ 2's, $\dots$, and $\alpha_l$ $l$'s.
For $1 \leq i \leq l$, define $\lambda^{[i]} = (1^{\alpha_1}, 2^{\alpha_2}, \dots, i^{\alpha_i})$ and $d_i(\lambda) = d(\lambda^{[i]})$.  
For example, if $\lambda = (2^3, 4^1) = (4,2,2,2)$, then $d(\lambda) = 2$ and $d_1(\lambda) = 0$, 
$d_2(\lambda) = d_3(\lambda) = 1$, $d_4(\lambda) = 2$.

For skew-symmetric forms, the degeneracy number of a partition $\lambda$ is defined to be the number of odd parts which 
appear odd number of times in $\lambda$. The rest of the definition is modified accordingly. 

\end{Definition}

\begin{Lemma}\label{lem:rank upper bound}
Let $M$ be the generic element of $\Sym_n^{\lambda}$. With respect to the vertical lines in its $\lambda$-decomposition,
let $M'$ be the matrix obtained by taking, in each column block, only the first $d_i(\lambda)$ columns of $M$ 
if the block corresponds to $i$. Then the null space of $M'$ has dimension $d(\lambda)$. 
In particular, $\text{corank}(M) \geq d(\lambda)$. The same claims are true for skew-symmetric forms.
\end{Lemma}

Before we start the proof let us illustrate the construction of $M'$. 
\begin{Example}\label{ex:rank}
Let $\lambda = (4,2,2,2)$. Then $d(\lambda) =2$. In this case, the generic element of $\Sym_{10}^{(4,2,2,2)}$ is
\begin{equation*}
M = \left( \begin{array}{ccccccccccccc}
0 & 0 & 0 & 0 & \vline & 0 & 0 & \vline & 0 & 0 & \vline & 0 & 0 \\
0 & 0 & 0 & b & \vline & 0 & 0 & \vline & 0 & 0 & \vline & 0 & 0 \\
0 & 0 & -b & 0 & \vline & 0 & d & \vline & 0 & g & \vline & 0 & l \\
0 & b & 0 & a & \vline & -d & c & \vline & -g & f & \vline & -l & k \\
\hline
0 & 0 & 0 & -d & \vline & 0 & 0 & \vline & 0 & i & \vline & 0 & n \\
0 & 0 & d & c & \vline & 0 & e & \vline & -i & h & \vline & -n & m \\
\hline
0 & 0 & 0 & -g & \vline & 0 & -i & \vline & 0 & 0 & \vline & 0 & q \\
0 & 0 & g & f & \vline & i & h & \vline & 0 & j & \vline & -q & p \\
\hline
0 & 0 & 0 & -l & \vline & 0 & -n & \vline & 0 & -q & \vline & 0 & 0 \\
0 & 0 & l & k & \vline & n & m & \vline & q & p & \vline & 0 & r
\end{array} \right).
\end{equation*}
Since the first block-column has four columns and since $d_4(\lambda) = 2$,
we take its first two columns. On the other hand, each of the second, third and the fourth block-columns have two columns. 
Since $d_2 (\lambda) = 1$, we take the first column of each. Therefore,
\begin{equation*}
M' = \left( \begin{array}{ccccccccccccc}
0 & 0 &\vline & 0 &\vline & 0 & \vline & 0  \\
0 & 0  &\vline & 0 &\vline & 0  & \vline & 0  \\
0 & 0  &\vline & 0 & \vline & 0 & \vline & 0 \\
0 & b  &\vline & -d &  \vline & -g & \vline & -l \\
\hline
0 & 0  &\vline & 0  &\vline & 0 & \vline & 0  \\
0 & 0  &\vline & 0  &\vline & -i &  \vline & -n \\
\hline
0 & 0 &\vline & 0 &\vline & 0 & \vline & 0  \\
0 & 0 &\vline & i  &\vline & 0 & \vline & -q  \\
\hline
0 & 0 & \vline & 0 & \vline & 0 &  \vline & 0  \\
0 & 0 &\vline & n  & \vline & q &  \vline & 0 
\end{array} \right).
\end{equation*}
Notice that the rank of $M'$ is 3. Since $M'$ is a linear map from $\mb{K}^5$ to $\mb{K}^8$, the dimension of its 
null space is $\dim \mb{K}^5 - \mt{rank}(M') = 5 -3 =2$.

\end{Example}

\begin{proof}
We argue by induction on $d(\lambda)$.
In order to make the induction work, we prove a slightly more complicated statement.
We use the horizontal lines in the $\lambda$-decomposition
to form another matrix $M''$ by taking, in each row block of $M'$ corresponding to the part $i$,
only the last $d_i(\lambda)$ rows. For example, for $M$ as in Example \ref{ex:rank}, 
\begin{equation}\label{E: M - double prime}
M'' = \left( \begin{array}{ccccccccccccc}
0 & 0  &\vline & 0 & \vline & 0 & \vline & 0 \\
0 & b  &\vline & -d &  \vline & -g & \vline & -l \\
\hline
0 & 0  &\vline & 0  &\vline & -i &  \vline & -n \\
\hline
0 & 0 &\vline & i  &\vline & 0 & \vline & -q  \\
\hline
0 & 0 &\vline & n  & \vline & q &  \vline & 0 
\end{array} \right).
\end{equation}

The statements we prove are:

{\it  
\begin{enumerate}
\item There are $d(\lambda)$ linearly independent row relations in $M''$ that can be described explicitly in the following sense.
Let $m_i$ denote the number of even parts $\leq i$ that occur an odd number of times in $\lambda$.
Then, for each even part $i$ that occurs an odd number of times, there is a relation among the $m_i$-th rows from the bottom 
of the block-rows associated with $i$. 

\item The null spaces of $M'$ and $M''$ are the same, of dimension $d(\lambda)$. 
\end{enumerate}

The corresponding statements for skew-symmetric forms are similar. Consider the number $m_i$ of odd parts $\leq i$ 
that occur an even number of times. Then, for each odd part $i$ that occurs an odd number of times, 
there is a relation among the $m_i$-th rows from the bottom of the blocks corresponding to the part $i$ of $\lambda$. Furthermore,
the null spaces of $M$ and $M''$ are the same, of dimension $d(\lambda)$.
}

Let us illustrate the first claim for $M''$ as in (\ref{E: M - double prime}).
There are three block-rows corresponding to $i=2$, each of which has only one row. Since 
$m_2= 1$, this gives a relation among the third, fourth and the fifth rows of $M''$. 
For $i=4$, there is a single block-row in $M''$ with two rows and $m_4=2$. 
Therefore, there is a single relation for the first row of $M''$, which implies that the row is zero.

Now we proceed with the inductive argument. Note that everything to be proved is now in terms of the smaller matrix $M''$.
Let $s$ be the smallest even part that occurs an odd number of times.
We coarsen the block decomposition of $M''$ into simply
$$
M'' = \left( \begin{array}{ccc}
U'' & \vline & V'' \\
\hline
Z'' & \vline & W''
\end{array} \right).
$$
The lines divide between the blocks corresponding to parts $> s$ and those corresponding to parts  $\leq s$. 
(Above the horizontal line are the blocks corresponding to the parts $>s$.)
For example, $s=2$ is the smallest even part that appears an odd number of times in $\lambda = (4,2,2,2)$. Therefore, 
the corresponding $M''$ is given by 
\begin{equation*}
M'' = \left( \begin{array}{ccccccccccccc}
0 & 0  &\vline & 0 & 0 &  0 \\
0 & b  &\vline & -d & -g & -l \\
\hline
0 & 0  &\vline & 0  & -i &  -n \\
0 & 0 &\vline & i  & 0 & -q  \\
0 & 0 &\vline & n  & q & 0 
\end{array} \right).
\end{equation*}
For a skew-symmetric form the coarsening of $M''$ is obtained by considering the smallest odd part 
which appears odd number of times.

We claim that $Z'' = 0$ in general.
To see this, note that every block in $Z''$ consists of the first $d_j(\lambda)$ columns of a matrix $B_{j,k}^\mathsf{T}$
for some $j > i \geq k$. Since the first $j - k$ columns of $B_{j,k}^\mathsf{T}$ are zero, and 
since $d_j(\lambda) \leq j - i \leq j -k$, we find $Z'' = 0$.

Let us split the partition $\lambda$ into two sub-partitions $\lambda_U$ and $\lambda_W$ in accordance with the 
coarsening of $M''$, as follows; $\lambda_U := ((s+1)^{\alpha_{s+1}}, (s+2)^{\alpha_{s+2}}, \dots, l^{\alpha_l})$
and $\lambda_W := \lambda^{[s]} = (1^{\alpha_1}, 2^{\alpha_2}, \dots, s^{\alpha_s})$.
Denote by $U$ the generic element of $\Sym_{|\lambda_U|}^{\lambda_U}$, and similarly, denote by $W$ the generic element of 
$\Sym_{| \lambda_W|}^{\lambda_W}$. Observe that the blocks $U''$ and $W''$ of $M''$ 
are formed from $U$ and $W$ in the same manner as $M''$ is formed from $M$.
Since, $d(\lambda_U) = d(\lambda) - 1$ and $d(\lambda_W) = 1$,
we apply the inductive hypothesis to $U''$ and $W''$.

We first prove the claim about the row relations in $M''$.
There is a single relation among the last rows of the row-blocks in $W''$ that come from the part $s$.  
Since $Z'' = 0$, this gives a corresponding row relation in $M''$.
For each even part $j$ occurring an odd number of times with $j>s$, there are row relations in the bottom 
$(m_j - 1)$-th rows of the row-blocks of the submatrix $U''$. 
However, because of the form of the matrices $B_{p,q}$ (c.f. (\ref{E: general form of a unipotent fixed matrix}))
the corresponding rows in $V''$ consists of zeros only. In other words, the row relations of $U''$ extend to 
row relations for $M''$. Together with the other relation found above, this gives $d(\lambda)$ 
linearly independent row relations with the claimed form.

Now we proceed to prove the second claim. 
The statement that the null space of $M'$ is the same as that of $M''$ is proven directly, without induction. Indeed, 
looking at (\ref{E: general form of a unipotent fixed matrix}), the general form of a unipotent fixed matrix, we see that 
in each block there are no non-zero entries above the main antidiagonal. Therefore, the rows in $M''$ that are deleted 
in order to obtain $M'$ are all zero rows. This proves the claim.

To calculate the rank of $M''$ first notice that there are $d(\lambda) - 1$ independent column 
relations in $U''$ and since $Z'' = 0$, this produces $d(\lambda) - 1$ independent column relations in $M''$.
There is also another column relation among the first columns in each block of $W''$ corresponding to the part $s$.
Let us call the columns occurring in this last relation {\it distinguished columns}.

Of course, the corresponding columns in $V''$ are {\it not} zero, so we cannot immediately extend this relation to one in $M''$.
Instead, we show that we can use certain additional columns of $M''$ to obtain a column relation.
To do this, it suffices to show that all of the distinguished columns of $V''$ lie in the column space of $U''$.
For then we can add a linear combination of columns in $U''$ to the linear combination of distinguished columns in $V''$ to produce zero.
The same linear combination of the full columns in $M''$ (obtained by simply adding $0$'s at the bottom) plus the combination of distinguished columns in $M''$ will produce zero as well.
This yields an independent column relation, giving a total of $d(\lambda)$ linearly independent column relations in $M''$.
In other words, we have $d(\lambda)$ linearly independent vectors in the null space of $M''$.

To prove that the distinguished column space of $V''$ is contained in the column space of $U''$, it is enough to show that the distinguished column space of $V''$ is orthogonal to the kernel of $U''^{\mathsf{T}}$, which we interpret as the space of row relations in $U''$.
Now since the row relations in $U''$ always involve the bottom $m_j$-th rows from the various blocks and each $m_j \geq 2$,
while the nonzero entires in the distinguished columns of $V''$ are always located in the first rows of each block, the claim follows immediately.

\end{proof}

\begin{Lemma}\label{lem:rank lower bound}
Let $M$ be the generic element of $\Sym_n^{\lambda}$.
There exists a non-zero minor of size $(n - d(\lambda))$-by-$(n - d(\lambda))$.  
In particular, $\text{corank}(M) \leq d(\lambda)$. The same claim is true for skew-symmetric forms. 
\end{Lemma}

\begin{proof}
	
We prove the existence of a non-zero minor of the specified size
by finding a non-zero monomial  term in the minor expansion that occurs only once,
so that no cancellation can occur.
To do this, we use a slightly weaker decomposition of $M$ than its $\lambda$-decomposition.  
In the $\lambda$-decomposition of $M$,
remove any horizontal and vertical lines  that divide two equal parts of $\lambda$.
We then use the diagonal blocks of this decomposition to prove the result; 
since the variables in each block are distinct,
it suffices to prove the corresponding result for a single such diagonal block.
	
If the part of $\lambda$ is odd, then it is easy to see that all of the main antidiagonal terms are nonzero 
and that their product is a desired monomial.  (For a skew-symmetric $M$, if the part of $\lambda$ is even, then the product of the 
main antidiagonal terms gives the desired monomial.)

Similarly, if the part of $\lambda$ is even and occurs an even number of times,
then all of the main antidiagonal terms are nonzero and their product is the desired monomial.
(For a skew-symmetric $M$, if the part of $\lambda$ is odd and repeated even number times, then 
then all of the main antidiagonal terms are nonzero and their product is the desired monomial.)

On the other hand, if the part of $\lambda$ is even and occurs an odd number of times,
then all of the main antidiagonal terms are nonzero except those in the middle block.
But the antidiagonal terms just above the main antidiagonal of the middle block are nonzero,
so the product of all of these entries gives the desired monomial,
proving that this matrix has corank at most 1.
(The same argument works for a skew-symmetric $M$, if the part of $\lambda$ is odd and repeated odd number 
of times.)

\end{proof}

\begin{Example}

We continue with Example \ref{ex:rank}.

\begin{equation*}
M = \left( \begin{array}{ccccccccccccc}
0 & 0 & 0 & 0 & \vline & 0 & 0 & \vline & 0 & 0 & \vline & 0 & 0 \\
0 & 0 & 0 & b & \vline & 0 & 0 & \vline & 0 & 0 & \vline & 0 & 0 \\
0 & 0 & -b & 0 & \vline & 0 & d & \vline & 0 & g & \vline & 0 & l \\
0 & b & 0 & a & \vline & -d & c & \vline & -g & f & \vline & -l & k \\
\hline
0 & 0 & 0 & -d & \vline & 0 & 0 & \vline & 0 & i & \vline & 0 & n \\
0 & 0 & d & c & \vline & 0 & e & \vline & -i & h & \vline & -n & m \\
\hline
0 & 0 & 0 & -g & \vline & 0 & -i & \vline & 0 & 0 & \vline & 0 & q \\
0 & 0 & g & f & \vline & i & h & \vline & 0 & j & \vline & -q & p \\
\hline
0 & 0 & 0 & -l & \vline & 0 & -n & \vline & 0 & -q & \vline & 0 & 0 \\
0 & 0 & l & k & \vline & n & m & \vline & q & p & \vline & 0 & r
\end{array} \right).
\end{equation*}

The matrix $M$ has rank $8$, with column relations
\begin{equation}\label{eq:rank column relation 1}
(qd - ng + il)C_2 - bqC_5 + bnC_7 - biC_{9} = 0
\end{equation}
\begin{equation}\label{eq:rank column relation 2}
C_1 = 0,
\end{equation}
where $C_i$ denotes the $i$-th column of $M$.

In the proof of Lemma \ref{lem:rank upper bound}, 
we find the column relations in $U''$ and an additional column relation in $W''$.
The first column of $U''$ being zero implies (\ref{eq:rank column relation 2}).
There is a single relation among the columns of $W''$,
namely $q C^{W''}_1 - n C^{W''}_2 + i C^{W''}_3 = 0$.
Moreover, in this example, the column space of $V''$ is the same as that of $U''$, therefore assuring a relation among 
$C^{M''}_3$, $C^{M''}_4$, and $C^{M''}_5$.
Transporting that relation back to $M$ gives (\ref{eq:rank column relation 1}).

In the proof of Lemma \ref{lem:rank lower bound}, we consider the minors that uses rows and columns 2,3,4 as well as 
5,6,8,9,10.  The relevant monomials are $-b^3$ and $n^4j$. Thus, the combined minor of size 3+5=8 and 
it has determinant $-b^3n^4j$.
\end{Example}

\begin{Corollary}\label{cor:det 0}
The determinant of the generic element of $\Sym_n^{\lambda}$ is zero if and only if every even part 
which occurs in $\lambda$ occurs an even number of times. Similarly, the determinant of the generic element of 
$\Skew_n^{\lambda}$ is zero if and only if every odd part which occurs in $\lambda$ occurs an even number of times.
\end{Corollary}

\bibliography{References}
\bibliographystyle{plain}

\end{document}